\documentclass[12pt]{article}
\usepackage{fullpage,amsthm,amssymb,amsmath}
\usepackage{enumerate,color}
\usepackage{tikz}
\usetikzlibrary{matrix,arrows}
\usetikzlibrary{positioning}
\usetikzlibrary{fit}
\usetikzlibrary{patterns}
\usetikzlibrary{decorations.pathreplacing,automata}
\tikzstyle{vertex}=[circle,fill=black!100,text=white,inner sep=0.8mm]
\tikzstyle{pointV}=[circle,fill=black,inner sep=0.5mm]

\newtheorem{theorem}{Theorem}[section]
\newtheorem{lemma}[theorem]{Lemma}
\newtheorem{prop}[theorem]{Proposition}
\newtheorem{cor}[theorem]{Corollary}

\newtheorem{problem}{Problem}
\newtheorem{obs}[theorem]{Observation}

\newcommand\needed[1]{
}

\begin{document}

\title{On $k$-$11$-representable graphs}

\author{
Gi-Sang Cheon
\thanks{Applied Algebra and Optimization Research Center, Department of Mathematics, Sungkyunkwan University, Suwon, Republic of Korea.
\texttt{gscheon@skku.edu}
}
\and
Jinha Kim
\thanks{Department of Mathematical Sciences, Seoul National University, Seoul, Republic of Korea.
\texttt{kjh1210@snu.ac.kr}
}
\and
Minki Kim
\thanks{Department of Mathematical Sciences, KAIST, Daejeon, Republic of Korea.
\texttt{kmk90@kaist.ac.kr}
}
\and
Sergey Kitaev
\thanks{Department of Computer and Information Sciences, University of Strathclyde, Glasgow, United Kingdom.
\texttt{sergey.kitaev@cis.strath.ac.uk}
}
\and
Artem Pyatkin
\thanks{Sobolev Institute of Mathematics, Novosibirsk State University, 630090 Novosibirsk, Russia.
\texttt{artem@math.nsc.ru}
}
}

\date\today

\maketitle

\begin{abstract}
Distinct letters $x$ and $y$ alternate in a word $w$ if after deleting in $w$ all letters but the copies of $x$ and $y$ we either obtain a word of the form $xyxy\cdots$ (of even or odd length) or a word of the form $yxyx\cdots$ (of even or odd length). A simple graph $G=(V,E)$ is word-representable if there exists a word $w$ over the alphabet $V$ such that letters $x$ and $y$ alternate in $w$ if and only if $xy$ is an edge in $E$. 
Thus, edges of $G$ are defined by avoiding the consecutive pattern 11 in a word representing $G$, that is, by avoiding $xx$ and $yy$. 

In 2017, Jeff Remmel has introduced the notion of a $k$-$11$-representable graph for a non-negative integer $k$, which generalizes the notion of a word-representable graph. Under this representation, edges of $G$ are defined by containing at most $k$ occurrences of the consecutive pattern $11$ in a word representing $G$. Thus, word-representable graphs are precisely $0$-$11$-representable graphs. Our key result in this paper is showing that any graph is $2$-$11$-representable by a concatenation of permutations, which is rather surprising taking into account that concatenation of permutations has limited power in the case of $0$-$11$-representation. Also, we show that the class of word-representable graphs, studied intensively in the literature, is contained strictly in the class of $1$-$11$-representable graphs. Another result that we prove is the fact that the class of interval graphs is precisely the class of $1$-$11$-representable graphs that can be represented by uniform words containing two copies of each letter. This result can be compared with the known fact that the class of circle graphs is precisely the class of $0$-$11$-representable graphs that can be represented by uniform words containing two copies of each letter. \\


\end{abstract}

\maketitle

\section{Introduction}\label{intro}

The theory of word-representable graphs is a young but very promising research area.  It was introduced by the forth author in 2004 based on the joint research with Steven Seif \cite{KS} on the celebrated {\em Perkins semigroup}, which has played a central role in semigroup theory since 1960, particularly as a source of examples and counterexamples. However, the first systematic study of word-representable graphs was not undertaken until the appearance in 2008 of \cite{KP}, which started the development of the theory. 

Up to date, about 20 papers have been written on the subject, and the core of the book \cite{KL} is devoted to the theory of word-representable graphs. It should also be mentioned that the software packages \cite{Glen,Hans} are often of great help in dealing with word-representation of graphs. Moreover, a recent paper~\cite{K} offers a comprehensive introduction to the theory. Some motivation points to study these graphs are given in Section~\ref{intro}. 

%

A simple graph $G=(V,E)$ is {\em word-representable} if and only if there exists a word $w$ over the alphabet $V$ such that letters $x$ and $y$, $x\neq y$, alternate in $w$ if and only if $xy\in E$. In other words, $xy\in E$ if and only if the subword of $w$ induced by $x$ and $y$ avoids the {\em consecutive pattern} 11 (which is an occurrence of $xx$ or $yy$). Such a word $w$ is called $G$'s {\em word-representant}. In this paper we assume $V$ to be $[n]=\{1,2,\ldots,n\}$ for some $n\geq 1$. For example, the cycle graph on 4 vertices labeled by 1, 2, 3 and 4 in clockwise direction can be represented by the word 14213243. Note that a complete graph $K_n$ can be represented by any  permutation of $[n]$, while an edgeless graph (i.e.\ empty graph) on $n$ vertices can be represented by $1122\cdots nn$. 
 Not all graphs are word-representable, and the minimum non-word-representable graph is the wheel graph $W_5$ in Figure~\ref{W5-fig}, which is the only non-word-representable graph on six vertices~\cite{KL,KP}. 

In 2017, Jeff Remmel \cite{Remmel} has introduced the notion of a $k$-$11$-representable graph for a non-negative integer $k$, which generalizes the notion of a word-representable graph. Under this representation, edges of $G$ are defined by containing at most $k$ occurrences of the consecutive pattern 11 in a word representing $G$. Thus, word-representable graphs are precisely $0$-$11$-representable graphs. The new definition allows to 
\begin{itemize} 
\item represent {\em any} graph; Theorem~\ref{all2-11} shows that any graph is $2$-$11$-representable by a concatenation of permutations, which is rather surprising taking into account that concatenation of permutations has limited power in the case of $0$-$11$-representation (see Theorem~\ref{comp-thm}). $2$-$11$-representation could be compared with the possibility to {\em $u$-represent any} graph, where $u\in\{1,2\}^*$ of length at least 3 \cite{Kit1}. We refer the Reader to  \cite{Kit1} for the relevant definitions just mentioning that the case of $u=11$ corresponds to word-representable graphs.
\item $1$-$11$-represent at least some of non-word-representable graphs including $W_5$ and all such graphs on seven vertices (see Section~\ref{repr-non-repr-sec}).
\item give a new characterization of interval graphs; see Theorem~\ref{interval-thm}, which should be compared with Theorem~\ref{circle-class-thm} characterizing circle graphs.
\end{itemize}

The paper is organized as follows. In the rest of the section, we give more details about word-representable graphs. In Section~\ref{sec-2}, we introduce rigorously the notion of a $k$-$11$-representable graph and provide a number of general results on these graphs. In particular, we show that a $(k-1)$-$11$-representable graph is necessarily $k$-$11$-representable (see Theorem~\ref{basic-inclusion}). In Section~\ref{sec3}, we study the class of 1-$11$-representable graphs. These studies are extended in Section~\ref{repr-non-repr-sec}, where we 1-11-represent all non-word-representable graphs on at most 7 vertices. In Section~\ref{sec-2-11-representable} we prove that any graph is $2$-$11$-representable.  Finally, in Section~\ref{open-prob-sec}, we state a number of open problems on $k$-$11$-representable graphs. 


Motivation points to study word-representable graphs include the fact exposed in \cite{KL} that these graphs generalize several important classes of graphs such as {\em circle graphs} \cite{Cer}, 3-{\em colourable graphs} and {\em comparability graphs} \cite{Lov}. Relevance of word-representable graphs to scheduling problems was explained in \cite{HKP} and it was based on \cite{GZ}. Furthermore, the study of word-representable graphs is interesting from an algorithmic point of view as explained in \cite{KL}. For example, the {\em Maximum Clique problem} is polynomially solvable on word-representable graphs \cite{KL} while this problem is generally NP-complete \cite{BBPP}. Finally, word-representable graphs are an important class among other graph classes considered in the literature that are defined using words. Examples of other such classes of graphs are {\em polygon-circle graphs} \cite{Koebe} and {\em word-digraphs} \cite{Bell}. 

The following two theorems are useful tools to study word-representable graphs. For the second theorem, we need the notion of a {\em cyclic shift} of a word. Let a word $w$ be the concatenation $uv$ of two non-empty words $u$ and $v$. Then, the word $vu$ is a cyclic shift of $w$.

\begin{theorem}[\cite{KP}]\label{unif-wr} A graph is word-representable if and only if it can be represented uniformly, i.e.\ using the same number of copies of each letter. \end{theorem}

\begin{theorem}[\cite{KP}]\label{cyclic-shift} Any cyclic shift of a word having the same number of copies of each letter represents the same graph.   \end{theorem}

A {\em circle graph} is the intersection graph of a set of chords of a circle, i.e.\ it is an undirected graph whose vertices can be associated with chords of a circle such that two vertices are adjacent if and only if the corresponding chords cross each other. In this paper, we get used of the following theorem.

\begin{theorem}[\cite{HKP}]\label{circle-class-thm} The class of circle graphs is precisely the class of word-representable graphs that can be represented by uniform words containing two copies of each letter. \end{theorem}

An orientation of a graph is {\em transitive}, if the presence of the edges $u\rightarrow v$ and $v\rightarrow z$ implies the presence of the edge $u\rightarrow z$. An oriented graph $G$ is a {\em comparability graph} if $G$ admits a transitive orientation. 
A graph is {\em permutationally representable} if it can be represented by concatenation of permutations of (all) vertices. Thus, permutationally representable graphs are a subclass of word-representable graphs. The following theorem classifies these graphs.

\begin{theorem}[\cite{KS}]\label{comp-thm} A graph is permutationally representable if and only if it is a comparability graph. \end{theorem} 

\section{Definitions and general results}\label{sec-2}
A {\em factor} in a word $w_1w_2\ldots w_n$ is a word $w_iw_{i+1}\ldots w_j$ for $1\leq i\leq j\leq n$. For a letter or a word $x$, we let $x^k$ denote $\underbrace{x\ldots x}_{k\mbox{ \small times}}$. For any word $w$, we let $\pi(w)$ denote the {\em initial permutation} of $w$ obtained by reading $w$ from left to right and recording the leftmost occurrences of the letters in $w$. For example, if $w=2535214421$ then $\pi(w)=25314$. Similarly, the {\em final permutation} $\sigma(w)$ of $w$ is obtained by reading $w$ from right to left and recording the rightmost occurrences of $w$. For the $w$ above, $\sigma(w)=35421$. Also, for a word $w$, we let $r(w)$ denote the {\em reverse} of $w$, that is, $w$ written in the reverse order. For example, if $w=22431$ then $r(w)=13422$. Finally, for a pair of letters $x$ and $y$ in a word $w$, we let $w|_{\{x,y\}}$ denote the word induced by the letters $x$ and~$y$. For example, for the word $w=2535214421$, $w|_{\{2,5\}}=25522$. The last definition can be extended in a straightforward way to defining $w|_S$ for a set of letters $S$. For example, for the same $w$, $w|_{\{1,2,3\}}=232121$. 

Throughout this paper, we denote by $G\setminus v$ the graph obtained from a graph $G$ by deleting a vertex $v\in V(G)$ and all edges adjacent to it. 

Let $k\geq 0$. A graph $G=(V,E)$ is {\em $k$-$11$-representable} if there exists a word $w$ over the alphabet $V$ such that the word $w|_{\{x,y\}}$  contains in total at most $k$ occurrences of the factors in $\{xx,yy\}$ if and only if $xy$ is an edge in $E$. Such a word $w$ is called $G$'s {\em $k$-$11$-representant}. A {\em uniform} (resp., {\em $t$-uniform}) representation of a graph $G$ is a word, satisfying the required properties, in which each letter occurs the same (resp., $t$) number of times. As is stated above, in this paper we assume $V$ to be $[n]=\{1,2,\ldots,n\}$ for some $n\geq 1$. Note that $0$-$11$-representable graphs are precisely word-representable graphs, and that $0$-$11$-representants are precisely word-representants. We also note that the ``11'' in ``$k$-$11$-representable'' refers to counting occurrences of the {\em consecutive pattern} 11 in the word induced by a pair of letters $\{x,y\}$, which is exactly the total number of occurrences of the factors in $\{xx,yy\}$.  Throughout the paper, we normally omit the word ``consecutive'' in ``consecutive pattern'' for brevity. Finally, we let $\mathcal{G}^{(k)}$ denote the class of $k$-$11$-representable graphs.  

\begin{lemma}\label{extension-lem} Let $k\geq 0$ and a word $w$ $k$-$11$-represent a graph $G$. Then the word $r(\pi(w))w$ $(k+1)$-$11$-represents $G$. Also, the word $wr(\sigma(w))$ $(k+1)$-$11$-represents $G$. Moreover, if $k=0$ then the word $ww$ $1$-$11$-represents $G$. \end{lemma}

\begin{proof} Suppose $x$ and $y$ are two vertices in $G$. If $xy$ is an edge in $G$ then $w|_{\{x,y\}}$ contains at most $k$ occurrences of the pattern $11$, so $(r(\pi(w))w)|_{\{x,y\}}$ (resp., $(wr(\sigma(w)))|_{\{x,y\}}$) contains at most $k+1$ occurrences of the pattern 11, and $xy$ will be an edge in the new representation. On the other hand, if $xy$ is not an edge in $G$, then $w|_{\{x,y\}}$ contains at least $k+1$ occurrences of the pattern $11$, so $(r(\pi(w))w)|_{\{x,y\}}$ (resp., $(wr(\sigma(w)))|_{\{x,y\}}$) contains at least $k+2$ occurrences of the pattern 11, and $xy$ will not be an edge in the new representation.  

Finally, if $x$ and $y$ alternate in $w$, then $ww$ contains at most one occurrence of $xx$ or $yy$, while non-alternation of $x$ and $y$ in $w$ leads to at least two occurrence of the pattern 11 in $ww$, which involves $x$ or/and $y$. These observations prove the last claim.  \end{proof}

\begin{theorem}\label{basic-inclusion} We have $\mathcal{G}^{(k)}\subseteq \mathcal{G}^{(k+1)}$ for any $k\geq 0$. \end{theorem}

\begin{proof} This is an immediate corollary of Lemma~\ref{extension-lem}.\end{proof}

\begin{lemma}\label{begin-end} Let $k\geq 0$, $G$ be a $k$-$11$-representable graph, and $i$ and $j$ be vertices in $G$, possibly $i=j$. Then there are infinitely many  words $w$ $k$-representing $G$ such that $w=iw'j$ for some words $w'$.\end{lemma}

\begin{proof} Let $u$ $k$-represent $G$. Then note that any word $v$ of the form $\pi(u)\cdots \pi(u)u\sigma(u)\cdots \sigma(u)$ $k$-represents $G$. Deleting all letters to the left of the leftmost $i$ in $v$, and all letters to the right of the rightmost $j$ in $v$, we clearly do not change the number of occurrence of the pattern 11 for any pair of letters $\{x,y\}$. The obtained word $w$ satisfies the required properties.  \end{proof}

There is a number of properties that is shared between word-representable graphs and $k$-$11$-representable graphs for any $k\geq 1$. These properties can be summarized as follows:

\begin{itemize}
\item The class $\mathcal{G}^{(k)}$ is hereditary. Indeed, if a word $w$ $k$-$11$-represents a graph $G$, and $v$ is a vertex in $G$, then clearly the word  obtained from $w$ by removing $v$ $k$-$11$-represents the graph $G\backslash \{v\}$.
\item In the study of $k$-$11$-representable graphs, we can assume that graphs in question are connected (see Theorem~\ref{connected-thm}).
\item In the study of $k$-$11$-representable graphs, we can assume that graphs in question have no vertices of degree 1 (see Theorem~\ref{degree-1-thm}).
\item In the study of $k$-$11$-representable graphs, we can assume that graphs in question have no two vertices having the same neighbourhoods up to removing these vertices, if they are connected (see Theorem~\ref{same-neighbourhood-thm}).
\item Glueing two $k$-$11$-representable graphs in a vertex gives a $k$-$11$-representable graph (see Theorem~\ref{gluing-2-thm}).
\item Connecting two $k$-$11$-representable graphs by an edge gives a $k$-$11$-representable graph (see Theorem~\ref{connecting-2-thm}).
\end{itemize} 

\begin{theorem}\label{connected-thm} Let $k\geq 0$. A graph $G$ is $k$-$11$-representable if and only if each connected component of $G$ is $k$-$11$-representable. \end{theorem}

\begin{proof} If $G$ is $k$-$11$-representable then each of $G$'s connected components is $k$-$11$-representable by the hereditary property of $k$-$11$-representable graphs. 

Conversely, suppose that $C_i$'s are the connected components of $G$ for $1\leq i\leq\ell$, and $w_i$ $k$-$11$-represents $C_i$. Adjoining several copies of $\pi(w_i)$ to the left of $w_i$, if necessary, we can assume that each letter in any $w_i$ occurs at least $k+2$ times. But then, the word $w=w_1w_2\cdots w_{\ell}$ $k$-$11$-represents $G$, since 
\begin{itemize}
\item edges/non-edges in each $C_i$ are represented by the $w_i$, and
\item for $x\in C_i$ and $y\in C_j$, $i\neq j$, the word $w|_{\{x,y\}}$ contains at least $2k+2$ occurrences of the pattern 11 making $x$ and $y$ be disconnected in $G$, 
\end{itemize}
we are done.\end{proof}

\begin{theorem}\label{degree-1-thm} Let $k\geq 0$, $G$ be a graph with a vertex $x$, and $G_{xy}$ be the graph obtained from $G$ by adding to it a  vertex $y$ connected only to $x$. Then, $G$ is $k$-$11$-representable if and only if $G_{xy}$  is $k$-$11$-representable. \end{theorem}

\begin{proof} The backward direction follows directly from the hereditary nature of $k$-$11$-represent-ability. For the forward direction, suppose that $w$ $k$-$11$-represents $G$. Adjoining several copies of $\pi(w)$ to the left of $w$, if necessary, we can assume that $x$ occurs at least $2k+2$ times in $w$. Replacing every other occurrence of $x$ in $w$, starting from the leftmost one, with $yxy$, we obtain a word $w'$ that  $k$-$11$-represents $G_{xy}$. Indeed, clearly, the letters $x$ and $y$ alternate in $w'$ so $xy$ is an edge in $G_{xy}$ no matter what $k$ is. On the other hand, if $z\neq x$ is a vertex in $G$, then $w'|_{\{z,y\}}$ has at least $k+1$ occurrences of the pattern 11 (formed by $y$'s) ensuring that $zy$ is not an edge in $G_{xy}$. Any other alternation of letters in $w$ is the same as that in $w'$. \end{proof}

\begin{theorem}\label{same-neighbourhood-thm} Let $k\geq 0$ and $G$ be a graph having two, possibly connected vertices, $x$ and $y$, with the same neighbourhoods up to removing $x$ and $y$. Then, $G$ is $k$-$11$-representable if and only if $G\setminus x$  is $k$-$11$-representable.\end{theorem}

\begin{proof} The forward direction follows directly from the hereditary nature of $k$-$11$-represent-ability. For the backward direction, let $w$ $k$-$11$-represent $G\setminus x$. If $x$ and $y$ are connected in $G$, then replacing each $y$ by $xy$ in $w$ clearly gives a $k$-$11$-representant of $G$ because $x$ and $y$ will have the same properties and they will be strictly alternating. On the other hand, if $x$ and $y$ are not connected in $G$, then adjoining several copies of $\pi(w)$ to the left of $w$, if necessary, we can assume that $y$ occurs at least $k+2$ times in $w$. We then replace every even occurrence of $y$ in $w$ (from left to right) by $yx$, and every odd occurrence by $xy$. This will ensure that in the subword induced by $x$ and $y$, the number of occurrences of the pattern 11 is at least $k+1$ making $x$ and $y$ be not connected in $G$. On the other hand, still $x$ and $y$ have the same alternating properties with respect to other letters. Thus, the obtained word $k$-$11$-represents $G$, as desired. \end{proof}

\begin{theorem}\label{gluing-2-thm} Let $k\geq 0$, $G_1$ and $G_2$ be $k$-$11$-representable graphs, and the graph $G$ is obtained from $G_1$ and $G_2$  by identifying a vertex $x$ in $G_1$ with a vertex $y$ in $G_2$. Then, $G$ is $k$-$11$-representable. \end{theorem}

\begin{proof} 
Let $w_1$ and $w_2$ be $k$-$11$-representants of the graphs $G_1$ and $G_2$, respectively.
Recall that if a word $w$ $k$-$11$-represents a graph $H$, then the word $w' = \pi(w)w$ obtained from $w$ by adding the initial permutation $\pi(w)$ of $w$ in front of $w$ also $k$-$11$-represents $H$.
Applying this observation, we may assume that the number of occurrences of $x$ in the word $w_1$ equals to that of the letter $y$ in the word $w_2$.
In addition, by Lemma~\ref{begin-end}, we may further assume that $w_1$ starts with the letter $x$, and $w_2$ starts with the letter $y$.
That is, $w_1 = x g_1 x g_2 \dots x g_m$, where $g_i$'s are words over $V(G_1)\setminus\{x\}$, and $w_2 = y h_1 y h_2 \dots y h_m$, where $h_i$'s are words over $V(G_2)\setminus\{y\}$.
Let $\pi_1$ (resp., $\pi_2$) be the initial permutation of the word $g_1 g_2 \dots g_m$ (resp., $h_1 h_2 \dots h_m$). In other words, $\pi(w_1)=x \pi_1$ and $\pi(w_2)=y \pi_2$.

Let $z$ be the vertex in $G$ which corresponds to the vertices $x$ and $y$, i.e. $z=x=y$ in $G$. 
We claim that the word $w(G):=(z \pi_1 \pi_2 z \pi_2 \pi_1)^{k+1} z g_1 h_1 z g_2 h_2 \dots z g_m h_m$ $k$-$11$-represents the graph $G$.
The induced subword of $w(G)$ on $V(G_1)$ is precisely $\pi(w_1)^{2k+2} w_1$ which $k$-$11$-represents the graph $G_1$.
Similarly, the induced subword of $w(G)$ on $V(G_2)$ $k$-$11$-represents the graph $G_2$.
Now, consider $v_1 \neq x$ in $V(G_1)$ and $v_2 \neq y$ in $V(G_2)$.
By the definition of $G$, the vertices $v_1$ and $v_2$ are not adjacent in $G$. Thus, it remains to show that the induced subword $w(G)|_{\{v_1,v_2\}}$  has at least $k+1$ occurrences of the pattern 11, which is easy to see from $(v_1 v_2 v_2 v_1)^{k+1}$ being a factor of $w(G)|_{\{v_1,v_2\}}$.
Therefore, the word $w(G)$ indeed $k$-$11$-represents the graph $G$.
\end{proof}

\begin{theorem}\label{connecting-2-thm} Let $k\geq 0$, $G_1$ and $G_2$ be $k$-$11$-representable graphs, and the graph $G$ is obtained from $G_1$ and $G_2$  by connecting a vertex $x$ in $G_1$ with a vertex $y$ in $G_2$ by an edge. Then $G$ is $k$-$11$-representable.
\end{theorem}
\begin{proof}
Let $w_1$ and $w_2$ be $k$-$11$-representants of $G_1$ and $G_2$, respectively.
By the same argument as in Theorem~\ref{gluing-2-thm}, we can assume that the number of occurrences of the letter $x$ in the word $w_1$ equals that of the letter $y$ in the word $w_2$.
By Lemma~\ref{begin-end}, we can assume that $w_1$ begins with $x$, and $w_2$ ends with $y$.
In addition, we can assume that the initial permutation of $w_2$ ends with $y$.
Suppose the initial permutation of $w_2$ does not end with $y$, and let $A y B$ be the initial permutation.
It is clear that the word $w'_2 = B A y B w_2$ also $k$-$11$-represents $G_2$, so that we can consider $w'_2$ instead of $w_2$, and the initial permutation of $w_2'$ ends with $y$. 

Now we can write $w_1=x g_1 x g_2 \dots x g_m$, where $g_i$'s are words over $V(G_1)\setminus\{x\}$, and $w_2=h_1 y h_2 y \dots h_m y$, where $h_i$'s are words over $V(G_2)\setminus\{y\}$.
Let $\pi_1$ (resp., $\pi_2$) be the initial permutation of the word $g_1 g_2 \dots g_m$ (resp., $h_1 h_2 \dots h_m$). Observe that $\pi(w_1) = x\pi_1$ and $\pi(w_2)=\pi_2 y$.
We claim that the word $w(G):=(x \pi_1 \pi_2 y \pi_2 x y \pi_1)^{k+1} x g_1 h_1 y x g_2 h_2 y \dots x g_m h_m y$ is a $k$-$11$-representant of $G$.
As in Theorem~\ref{gluing-2-thm}, it is clear that the word $w(G)$ $k$-11-represents the graphs $G_1$ and $G_2$, when restricted to $V(G_1)$ and $V(G_2)$, respectively.
Also, $w(G)$ makes the vertices $x$ and $y$ be adjacent, because $w(G)|_{\{x,y\}}=(xy)^{2k+m+2}$.
Hence, it remains to show that for every $v_1 \in V(G_1)$ and $v_2 \in V(G_2)$ such that $v_1 \neq x$ or $v_2 \neq y$, which must be non-adjacent in $G$, the induced subword  $w(G)|_{\{v_1,v_2\}}$  has at least $k+1$ occurrences of the pattern 11.
This is obviously the case, because $w(G)|_{\{v_1,v_2\}}$ contains $(v_1v_2 v_2 v_1)^{k+1}$ having at least $2k+1$ occurrences of the pattern 11. 
Therefore, the word $w(G)$ $k$-$11$-represents the graph $G$.
\end{proof}

\begin{theorem}\label{rem-vert-gen-thm} Let $G$ be a graph with a vertex $v$. If $G\setminus v$ is $k$-uniform word-representable for $k\geq 1$, then $G$ is $(k-1)$-$11$-representable.
\end{theorem}
\begin{proof}
Let $w$ be a $k$-uniform word that represents the graph $G\setminus v$.
Let $N(v) \subset V(G)$ be the set of all neighbors of $v$ in $G$, and let $N^{c}(v)$ be the complement of $N(v)$ in $V(G) \setminus \{v\}$, i.e. $N^c(v) = V(G) \setminus (N(v)\cup\{v\})$. 
We will describe how to construct a $(k-1)$-$11$-representant $w(G)$ of $G$ from the word $w$.
Recall that $r(\pi(w))$ is the reverse of the initial permutation $\pi(w)$ of the word $w$.

We start with the word $\pi(w)|_{N(v)} v \pi(w)|_{N^{c}(v)}~w$, where $\pi(w)|_{N(v)}$ and $\pi(w)|_{N^{c}(v)}$ are the induced subwords of $\pi(w)$ on $N(v)$ and $N^c(v)$, respectively.
In each step, we adjoin the words $r(\pi(w))v$ and $\pi(w)v$, in turn, from the left side of the word constructed in the previous step.
We stop when the current word, denoted by $w(G)$, has exactly $k$ $v$'s.
For example, the word $w(G)$, when $k=6$, is given by
$$w(G)=r(\pi(w))v~\pi(w)v~r(\pi(w))v~\pi(w)v~r(\pi(w))v~\pi(w)|_{N(v)} v \pi(w)|_{N^{c}(v)}~w.$$
Next, we will show that the word $w(G)$ $(k-1)$-$11$-represents $G$. 
First, take a vertex $x \neq v$ in $G$.
If $x \in N(v)$, then $w(G)|_{\{x,v\}}=xv \dots xv~w|_{\{x\}}$ has $k-1$ occurrences of the pattern 11 since $w|_{\{x\}}=x^k$.
If $x \in N^c(v)$, then $w(G)|_{\{x,v\}}=xv \dots xv~vx~w|_{\{x\}}$ has $k+1$ occurrences of the pattern 11.
Thus $w(G)$ preserves all the (non-)adjacencies of $v$.
Now, take two distinct vertices, $y,z$ in $V(G)\setminus \{v\}$.
Without loss of generality, we can assume that $\pi(v)|_{\{y,z\}}=yz$.
If $y$ and $z$ are adjacent in $G\setminus v$, then $w|_{\{y,z\}}=yzyz \dots yz$.
Hence, the induced subword $$w(G)|_{\{y,z\}}=\dots zy~yz~zy~(\pi(w)|_{N(v)} v \pi(w)|_{N^{c}(v)})|_{\{y,z\}}~yzyz \dots yz$$ has $k-1$ occurrences of the pattern 11 since the part $\dots zy~yz~zy$ is of length $2(k-1)$, and $(\pi(w)|_{N(v)} v \pi(w)|_{N^{c}(v)})|_{\{y,z\}}$ is either $yz$ or $zy$.
If $y$ and $z$ are not adjacent in $G\setminus v$, then $w|_{\{y,z\}}$ has at least one occurrence of the pattern 11 and it starts with $y$.
Hence, $w(G)|_{\{y,z\}}=\dots zy~yz~zy~(\pi(w)|_{N(v)} v \pi(w)|_{N^{c}(v)})|_{\{y,z\}}~w|_{\{y,z\}}$ has at least $k$ occurrences of the pattern 11 since the only difference from the previous case is $w|_{\{y,z\}}$, which now has at least one occurrence of the pattern 11.
This proves that $w(G)$ is a $(k-1)$-$11$-representant of~$G$.
\end{proof}

\begin{theorem}\label{rem-vert-more-gen-thm}
For any non-negative integers $m$ and $k$ satisfying $2m-k-1 >0$, the following holds.
Let $G$ be a graph with a vertex $v$. If $G\setminus v$ is $m$-uniform $k$-$11$-representable, then $G$ is $(3m-k-1)$-uniform $(2m-2)$-$11$-representable.  \end{theorem}

\begin{proof}
Let $w$ be an $m$-uniform $k$-$11$-representant of $G\setminus v$, $N(v) \subset V(G)$ be the set of all neighbors of $v$ in $G$, and let $N^c(v) = V(G) \setminus (N(v)\cup\{v\})$.
We will describe how to construct a $(3m-k-1)$-uniform $(2m-2)$-$11$-representant $w(G)$ of $G$ from the word $w$.
Similarly to the proof of Theorem~\ref{rem-vert-gen-thm}, we start with the word $\pi(w)|_{N(v)} v \pi(w)|_{N^{c}(v)}~w$, and in each step, we adjoin $r(\pi(v))v$ and $\pi(w)v$, in turn, from the left side until $w(G)$ has exactly $2m-k-1$ occurrences of $v$.
Then, we adjoin $v^m$ from the left side. 
For example, when $k=3$ and $m=4$, the word $w(G)$ is given by
$$w(G)=vvvv~r(\pi(w))v~\pi(w)v~r(\pi(v))v~\pi(w)|_{N(v)} v \pi(w)|_{N^{c}(v)}~w.$$
It is easy to see that $w(G)$ is $(3m-k-1)$-uniform. Indeed, if
$x \in V(G) \setminus \{v\}$, then $w(G)$ contains $(2m-k-1)+m=3m-k-1$ $x$'s since $w$ is $m$-uniform; also, $w(G)$ contains $m+(2m-k-1)=3m-k-1$ $v$'s.
Next, we will show that $w(G)$ $(2m-2)$-$11$-represents $G$.

Let $x \in V(G) \setminus \{v\}$.
If $x \in N(v)$, then $w(G)|_{\{x,v\}}=v^m~xv \dots xv~x^m$.
Thus $w(G)|_{\{x,v\}}$ has $2m-2$ occurrences of the pattern 11.
If $x \in N^c (v)$, then the only difference from the previous case in $w(G)|_{\{x,v\}}$ is that $\pi(w)|_{N(v)} v \pi(w)|_{N^{c}(v)}$ is $vx$, not $xv$. Thus, $w(G)|_{\{x,v\}}$ has $2m$ occurrences of the pattern 11.
Now take two distinct vertices $x,y \in V(G) \setminus \{v\}$.
Without loss of generality, we can assume that $\pi(w)|_{\{x,y\}}=xy$.
If $x,y$ are adjacent in $G\setminus v$, then $w|_{\{x,y\}}$ has at most $k$ occurrences of the pattern 11.
Hence, 
$$w(G)|_{\{x,y\}}=\dots yx~xy~yx~(\pi(w)|_{N(v)} v \pi(w)|_{N^{c}(v)})|_{\{x,y\}}~w|_{\{x,y\}}.$$
Since the length of $\dots yx~xy~yx$ is $4m-2k-4$ and $(\pi(w)|_{N(v)} v \pi(w)|_{N^{c}(v)})|_{\{x,y\}}$ is $xy$ or $yx$, $w(G)|_{\{x,y\}}$ has at most $(2m-k-3)+1+k=2m-2$ occurrences of the pattern 11.
If $x,y$ are not adjacent in $G\setminus v$, then $w|_{\{x,y\}}$ has at least $k+1$ occurrences of the pattern 11.
In this case, the only difference from the previous case in $w(G)$ is $w|_{\{x,y\}}$ and so $w(G)|_{\{x,y\}}$ has at least $(2m-k-3)+1+k+1=2m-1$ occurrences of the pattern 11.
This proves that $w(G)$ is a $(2m-2)$-$11$-representant of~$G$.
\end{proof}

\begin{cor}\label{cor.lem-fund}
For any non-negative integers $n$ and $k$ satisfying $2n + k - 7 > 0$, if each graph on $n$ vertices is $(k+n-3)$-uniformly $k$-$11$-representable, then every graph on $n+1$ vertices is $(2k+3n-10)$-uniformly $(2k+2n-8)$-$11$-representable.
\end{cor}
\begin{proof}
This is a direct consequence of Theorem~\ref{rem-vert-more-gen-thm}.
Suppose every graph on $n$ vertices is $(k+n-3)$-uniformly $k$-$11$-representable, and $G$ is a graph on $n+1$ vertices.
Clearly, $k + n - 3$ is a positive integer since we have $2n + k - 7 > 0$.
Then for any vertex $v$ in $G$, the graph $G\setminus v$ obtained from $G$ by removing a vertex $v$ is $(k+n-3)$-uniformly $k$-$11$-representable.
Since $2(k+n-3)-k-1 = 2n+k-7 > 0$, we can apply Theorem~\ref{rem-vert-more-gen-thm}, concluding that the graph $G$ is $(2k+3n-10)$-uniform $(2k+2n-8)$-$11$-representable.
\end{proof}
\noindent In particular, Corollary~\ref{cor.lem-fund} holds for any integers $n \geq 5$ and $k \geq 0$.

\section{1-11-representable graphs}\label{sec3}
An {\em interval graph} has one vertex for each interval in a family of intervals, and an edge between every pair of vertices corresponding to intervals that intersect. Not all interval graphs are word-representable~\cite{KL}. However, all interval graphs are 1-$11$-representable using two copies of each letter, as shown in the following theorem. This shows that the notion of an interval graph admits a natural generalization in terms of 1-$11$-representable graphs (instead of $2$-uniform 1-$11$-representants, one can deal with $m$-uniform 1-$11$-representants for $m\geq 3$).

\begin{theorem}\label{interval-thm} A graph is an interval graph if and only if it is $2$-uniformly $1$-$11$-representable. \end{theorem}

\begin{proof} Let $G$ be a $1$-$11$-representable graph on $n$ vertices and $w=w_1w_2\ldots w_{2n}$ be a word that $2$-uniformly $1$-$11$-represents $G$. For any $v\in V(G)=[n]$, consider the interval $I_v=[v_1,v_2]$ on the real line such that $w_{v_1}=w_{v_2}=v$.  Note that $uv$ is an edge in $G$ if and only if $I_u$ and $I_v$ overlap. But then, $G$ is the interval graph given by the family of intervals $\{I_v:v\in [n]\}$. 

To see that any interval graph $G$ is necessarily $1$-$11$-representable, we note a well-known easy to see fact that in the definition of an interval graph, one can assume that overlapping intervals overlap in more than one point. But then, the endpoints of an interval $I_v$ will give the positions of the letter $v$ in a word $w$ constructed by recording relative positions of all the intervals. As above, one can see that such an $w$ 1-11-represents $G$. 
\end{proof}

Given a graph $G$ with an edge $xy$, we let $G^{\triangle}_{xy}$ be the graph obtained from $G$ by adding a vertex $z$ connected only to the vertices $x$ and $y$.  Thus, $G^{\triangle}_{xy}$ is obtained from $G$ by adding a triangle. If $G$ is word-representable, that is, $G\in \mathcal{G}^{(0)}$, then $G^{\triangle}_{xy}$ is not necessarily word-representable. This can be seeing on the  non-word-representable graph $D_1$ in Figure~\ref{25-non-repr}. Indeed, removing, for example, the top vertex in that graph, we obtain a word-representable graph, since the only non-word-representable graph on six vertices is the wheel $W_5$ \cite{KL,KP}. The following theorem establishes that adding a triangle is a safe operation in the case of 1-$11$-representable graphs.

\begin{theorem}\label{add-triangle-thm} Let $G\in \mathcal{G}^{(1)}$ and $xy$ be an edge in $G$. Then $G^{\triangle}_{xy}\in \mathcal{G}^{(1)}$. \end{theorem}
\begin{proof}
Let $w$ be an $1$-$11$-representant of $G$.
Note that, since $x$ and $y$ are adjacent in $G$, the letters $x$ and $y$ are either alternating in the word $w$, or $w|_{\{x,y\}}$ has exactly one occurrence of the pattern 11.
In each case, we will construct a word $\tilde{w}$ over $V(G^{\triangle}_{xy})$, which $1$-$11$-represents the graph $G^{\triangle}_{xy}$.

\begin{enumerate}[{\bf \text{Case }1.}]
\item Suppose that $x$ and $y$ are alternating in $w$.
By Lemma~\ref{begin-end}, we can assume that $w$ starts with $x$ and ends with $y$, i.e.\   $w=x~g_1~y~g_2 \dots x~g_m~y$, where $g_i$ is a word on $V(G)\setminus\{x,y\}$.
Also, we can assume that $m \geq 3$; if not, adjoin the initial permutation $\pi(w)$ to the left of $w$.
Now, we claim that the word 
$$\tilde{w}:=zxz~g_1~y~g_2~x~g_3~zyz~g_4~x~g_5~yz~g_6 \dots x~g_m~yz$$
$1$-$11$-represents the graph $G^{\triangle}_{xy}$, where $z \in V(G^{\triangle}_{xy})\setminus V(G)$.

It is clear that $\tilde{w}$ respects the whole structure of $G$ since the restriction of $\tilde{w}$ to $V(G)$ is $w$.
Since $\tilde{w}|_{\{x,z\}}=zxzxzzxz \dots xz$ and $\tilde{w}|_{\{y,x\}}=zzyzyzyz \dots yz$, $z$ is adjacent to $x$ and $y$.
On the other hand, for each $v \in V(G)\setminus\{x,y\}$, it is obvious that the induced subword $\tilde{w}|_{\{v,z\}}$ has at least two occurrences of the pattern 11, hence $z$ is not adjacent to $v$.
Therefore, $\tilde{w}$ $1$-$11$-represents the graph $G^{\triangle}_{xy}$.

\item Suppose $w|_{\{x,y\}}$ has exactly one occurrence of the pattern 11.
Without loss of generality, we can assume that $w|_{\{x,y\}}$ contains the occurrence of the factor $yy$.
By Lemma~\ref{begin-end}, we can also assume that $w$ starts with $x$ and ends with $x$, i.e.\ $$w=x~g_1~y~g_2 \dots x~g_{m-1}~y~g_{m}~y~h_1~x~h_2 \dots y~h_{l}~x$$ for some positive integers $m,l$, and words $g_i,h_j$ on $V(G)\setminus\{x,y\}$.
We claim that the word
$$\tilde{w}:=zxz~g_1~y~g_2~xz~g_3~y~g_4 \dots xz~g_{m-3}~y~g_{m-2}~x~g_{m-1}~zyz~g_{m}~y~h_1~xz~h_2 \dots y~h_{l}~xz$$
$1$-$11$-represents the graph $G^{\triangle}_{xy}$.

It is clear that $\tilde{w}$ respects the whole structure of $G$ since the restriction of $\tilde{w}$ to $V(G)$ is $w$.
Since $\tilde{w}|_{\{x,z\}}=zxzx \dots zxzzxz \dots xz$ and $\tilde{w}|_{\{y,z\}}=zzyzyz \dots yz$, $z$ is adjacent to $x$ and $y$.
On the other hand, for each $v \in V(G) \backslash \{x,y\}$, the induced subword $\tilde{w}|_{\{v,z\}}$ has at least two occurrences of the pattern 11, hence $z$ is not adjacent to $v$.
Therefore, $\tilde{w}$ $1$-$11$-represents the graph $G^{\triangle}_{xy}$.
\end{enumerate} 
\end{proof}

For the next theorem, Theorem~\ref{rem-vert-perm-repr-thm}, recall the definition of a permutationally representable graph in Section~\ref{intro}. Note that the proof of Theorem~\ref{rem-vert-perm-repr-thm} is similar to that of Theorem~\ref{rem-vert-gen-thm},   while Theorem~\ref{rem-vert-perm-repr-thm} deals with a stricter assumption. However, the stricter assumption is compensated by a stronger conclusion, justifying us having Theorem~\ref{rem-vert-gen-thm}.  

\begin{theorem}\label{rem-vert-perm-repr-thm} 
Let $G$ be a graph with a vertex $v$. If $G\setminus v$ is permutationally representable (equivalently, by Theorem~\ref{comp-thm}, if $G\setminus v$ is a comparability graph) then $G$ is $1$-$11$-representable.
\end{theorem}
\begin{proof}
Let $w$ be a $0$-$11$-representant of $G\setminus v$.
Since $G\setminus v$ is permutationally representable, we can assume that $w$ is of the form $w=\pi_1 \pi_2 \dots \pi_k$ for some positive integer $k$ and permutations $\pi_i$ of $V(G\setminus v)$.
Let $N(v)$ be the set of neighbours of $v$ in $G$ and let $N^{c}(v) := V(G)\setminus(N(v)\cup\{v\})$.
We claim that the word
$$w(G):=r(\pi(w))~v~\pi(w)|_{N(v)}~v~\pi(w)|_{N^c(v)}~\pi_1 v \pi_2 v \dots v \pi_k.$$ 
$1$-$11$-represents the graph $G$.

For each $x \in V(G) \setminus \{v\}$, if $x \in N(v)$ then the induced subword $w(G)|_{\{x,v\}}=xvxv \dots xvx$ is alternating, which should be the case.
If $x \in N^c(v)$, then the induced subword $w(G)|_{\{x,v\}}=xvvxxvxv \dots xvx$ has two occurrences of the pattern 11, which, again, should be the case.
Thus, $w(G)$ respects all adjacencies of the vertex $v$.
Now, take $y,z \in V(G) \setminus \{v\}$.
If $y$ and $z$ are adjacent in $G\setminus v$, then $w|_{\{y,z\}}$  has alternating $y$ and $z$.
Without loss of generality, assume that $w|_{\{y,z\}}=yzyz \dots yz$.
Then, the induced subword $w(G)|_{\{y,z\}}=zy~(\pi(w)|_{N(v)}~\pi(w)|_{N^c(v)})|_{\{y,z\}}~yzyz \dots yz$ has at most one occurrence of the pattern 11 as $(\pi(w)|_{N(v)}~\pi(w)|_{N^c(v)})|_{\{y,z\}}$ is either $yz$ or $zy$. 
If $y$ and $z$ are not adjacent in $G\setminus v$, then $w|_{\{y,z\}}$ is not alternating, i.e.\ it contains either $yy$ or $zz$.
Without loss of generality, assume that $w|_{\{y,z\}}$ contains $yy$. If $\pi(w)|_{\{y,z\}}=yz$, then with the assumption on an occurrence of $yy$, at least one occurrence of the factor $zz$ is not avoidable in $w$, so at least two occurrences of the pattern 11 in  $w(G)|_{\{y,z\}}$ are guaranteed. Otherwise,  $w|_{\{y,z\}}=zy \dots zy~yz \dots$.
Then, $w(G)|_{\{y,z\}}=yz~(\pi(w)|_{N(v)}~\pi(w)|_{N^c(v)})|_{\{y,z\}}~zy \dots zy~yz \dots $ has two occurrences of the pattern 11, as $(\pi(w)|_{N(v)}~\pi(w)|_{N^c(v)})|_{\{y,z\}}$ is either $yz$ or $zy$.
In any case, $w(G)$ preserves the (non-)adjacency of $y$ and $z$.
Therefore the word $w(G)$ $1$-$11$-represents the graph $G$.
\end{proof}

\begin{theorem}\label{edge-thm} Let $G$ be a word-representable graph and $e$ be an edge in $G$. Let $G\setminus e$ be the graph obtained from $G$ by removing $e$. Then, $G\setminus e$ is $1$-$11$-representable.
\end{theorem}

\begin{proof}
Let $e = xy$ and $w$ be $G$'s uniform word-representant that exists by Theorem~\ref{unif-wr}.
Without loss of generality, we can assume that $w|_{\{x,y\}}=xyxy\dots xy$.
We claim that the graph $G'$ on $V(G)$, which is $1$-$11$-represented by the word $w' := yxwwyx$, is precisely the graph $G\setminus e$.

It is clear that $x$ and $y$ are not adjacent in $G'$ since $w'|_{\{x,y\}}=yxxy\dots xyyx$.
Since the word $ww$ is a $1$-$11$-representant of $G$, it remains to show that for every vertex $z \in V(G)\setminus\{x,y\}$, and a vertex $i \in \{x,y\}$, $G'$ contains the edge $iz$ whenever $iz$ is an edge in $G$.
Suppose $iz$ is an edge in $G$.
Then, $ww|_{\{i,z\}}$ is either $iz\dots iz$, or $zi\dots zi$.
It follows that $w'|_{\{i,z\}}$ is either $iiz\dots izi$, or $izi\dots zii$. Thus, $iz$ is an edge in $G'$. If $iz$ is not an edge in $G$, then $ww|_{\{i,z\}}$ will contain at least two occurrences of the pattern 11, so $iz$ is not an edge in $G'$. This shows that $G' = G\setminus e$.
\end{proof}

The following two theorems generalize Theorem~\ref{edge-thm}. The reason that we keep Theorem~\ref{edge-thm} as a separate result is that it is very useful in 1-11-representing 25 non-word-representable graphs (see Section~\ref{repr-non-repr-sec}).

\begin{theorem}\label{clique-thm}
Let $G$ be a word-representable graph and $K$ be a vertex subset in $G$. Let $G_K$ be the graph obtained from $G$ by removing the edges $\{xy\in E(G):x,y\in K\}$. Then, $G_K$ is $1$-$11$-representable.
\end{theorem}
\begin{proof} 
Let $w$ be a uniform word-representant of $G$ that exists by Theorem~\ref{unif-wr}.
Let $p$ be the reverse of the initial permutation of $w|_K$, and let $q$ be the reverse of the final permutation of $w|_K$.
Note that if $K$ is a clique in $G$, then $p = q$.
It is straightforward to check that the word $w' := pwwq$ $1$-$11$-represents the graph $G_K$.
\end{proof}

\begin{theorem}\label{star-thm}
Let $G$ be a word-representable graph, $v$ be a vertex in $G$, and $N$ be a set of some (not necessarily all) neighbors of $v$ in $G$.
Let $G_N$ be the graph obtained from $G$ by removing the edges $\{uv : u\in N\}$.
Then, $G_N$ is $1$-$11$-representable.
\end{theorem}
\begin{proof} 
Let $N = \{v_1,\dots,v_k\}$ and $w$ be a uniform word-representant of $G$.
Since $w$ is uniform, by Theorem~\ref{cyclic-shift}, we can assume that $v$ is the first letter in $w$.
Without loss of generality, assume that $v_1 \dots v_k$ is the initial permutation of $w|_N$.
Then, it is easy to check that the word $w' := v_k \dots v_1 v ww v_k \dots v_1 v$ $1$-$11$-represents the graph $G_N$.
\end{proof}

%
%

\section{1-11-representing non-word-representable graphs}\label{repr-non-repr-sec}

All graphs on at most five vertices are word-representable, and there is only one non-word-representable graph, the wheel $W_5$, on six vertices (see Figure~\ref{W5-fig}). Also, there are 25 non-word-representable graphs on 7 vertices, which are shown in Figure~\ref{25-non-repr}. 

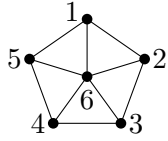
\begin{figure}[h]

\begin{center}
\begin{tikzpicture}[scale=0.9]
\draw (0,0)node[xshift=-0.2cm]{{\small 4}}node[pointV]{} --++(1,0)node[xshift=0.2cm]{{\small 3}}node[pointV]{}--++(0.35,0.95)node[xshift=0.2cm]{{\small 2}}node[pointV]{}
      --++(-0.85,0.59)node[yshift=0.1cm,xshift=-0.2cm]{{\small 1}}node[pointV]{}--++(-0.85,-0.59)node[xshift=-0.2cm]{{\small 5}}node[pointV]{}
      --++(0.35,-0.95)
      --++(0.5,0.69)node[yshift=-0.3cm]{{\small 6}}node[pointV]{}
      --+(0.5,-0.69)
      ++(0,0)--+(-0.85,0.26)
      ++(0,0)--+(0.85,0.26)
      ++(0,0)--++(0,0.85);

\end{tikzpicture}

\caption{The wheel graph $W_5$}\label{W5-fig}
\end{center}
\end{figure}

The following theorem shows that the notion of $k$-$11$-representability allows us to enlarge the class of word-representable graphs ($\mathcal{G}^{(0)}$), still by using alternating properties of letters in words.

\begin{theorem}\label{strict-inclusion-thm} We have $\mathcal{G}^{(0)}\subsetneq \mathcal{G}^{(1)}$. \end{theorem}

\begin{proof} By Theorem~\ref{basic-inclusion}, we have $\mathcal{G}^{(0)}\subseteq \mathcal{G}^{(1)}$. To show that the inclusion is strict, we give a word 1-11-representing the non-word-representable wheel graph $W_5$ in Figure~\ref{W5-fig}.  We start with 0-11-representing the cycle graph induced by all vertices but the vertex 6 by the 2-uniform word $w=1521324354$. This word, and a generic approach to find it, is found on page 36 in \cite{KL}. Note that the initial permutation $\pi(w)$ is $15234$, and thus, by Lemma~\ref{extension-lem}, the word $r(\pi(w))w=432511521324354$ 1-11-represents the cycle graph. Inserting a 6 in $w$ to obtain $u=4325161521324354$ gives a word 1-11-representing $W_5$ (which is easy to see). Note that the word $6u6$ gives a 3-uniform 1-$11$-representant of $W_5$.  \end{proof}

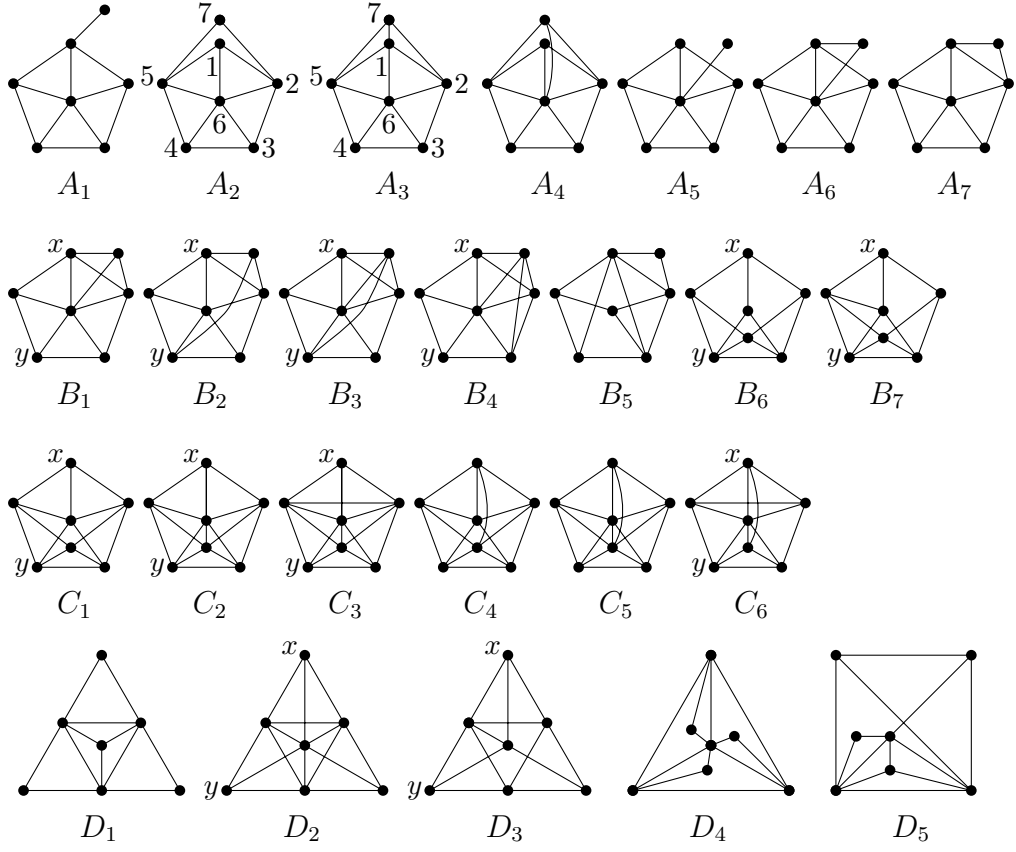
\begin{figure}[h]

\begin{center}
\begin{tikzpicture}[scale=0.9]
\draw (0,0)node[yshift=-0.5cm,xshift=0.5cm]{$A_1$}node[pointV]{} --++(1,0)node[pointV]{}--++(0.35,0.95)node[pointV]{}
      --++(-0.85,0.59)node[pointV]{}--++(-0.85,-0.59)node[pointV]{}
      --++(0.35,-0.95)
      --++(0.5,0.69)node[pointV]{}
      --+(0.5,-0.69)
      ++(0,0)--+(-0.85,0.26)
      ++(0,0)--+(0.85,0.26)
      ++(0,0)--++(0,0.85)
      --+(0.5,0.5)node[pointV]{};

\draw (2.2,0)node[yshift=-0.5cm,xshift=0.5cm]{$A_2$}node[xshift=-0.2cm]{{\small 4}}node[pointV]{} --++(1,0)node[xshift=0.2cm]{{\small 3}}node[pointV]{}--++(0.35,0.95)node[xshift=0.2cm]{{\small 2}}node[pointV]{}
      --++(-0.85,0.59)node[yshift=-0.3cm,xshift=-0.1cm]{{\small 1}}node[pointV]{}--++(-0.85,-0.59)node[yshift=0.1cm,xshift=-0.2cm]{{\small 5}}node[pointV]{}
      --++(0.35,-0.95)
      --++(0.5,0.69)node[yshift=-0.3cm]{{\small 6}}node[pointV]{}
      --+(0.5,-0.69)
      ++(0,0)--+(-0.85,0.26)
      --+(0,1.2)node[yshift=0.1cm,xshift=-0.2cm]{{\small 7}}node[pointV]{}
      ++(0,0)--+(0.85,0.26)
      --+(0,1.2)node[pointV]{}
      ++(0,0)--++(0,0.85)
      ;
\draw (4.7,0)node[yshift=-0.5cm,xshift=0.5cm]{$A_3$}node[xshift=-0.2cm]{{\small 4}}node[pointV]{} --++(1,0)node[xshift=0.2cm]{{\small 3}}node[pointV]{}--++(0.35,0.95)node[xshift=0.2cm]{{\small 2}}node[pointV]{}
      --++(-0.85,0.59)node[yshift=-0.3cm,xshift=-0.1cm]{{\small 1}}node[pointV]{}--++(-0.85,-0.59)node[yshift=0.1cm,xshift=-0.2cm]{{\small 5}}node[pointV]{}
      --++(0.35,-0.95)
      --++(0.5,0.69)node[yshift=-0.3cm]{{\small 6}}node[pointV]{}
      --+(0.5,-0.69)
      ++(0,0)--+(-0.85,0.26)
      --+(0,1.2)node[pointV]{}
      ++(0,0)--+(0.85,0.26)
      --+(0,1.2)node[yshift=0.1cm,xshift=-0.2cm]{{\small 7}}node[pointV]{}
      ++(0,0)--++(0,0.85)
      --+(0,0.35);
      
\draw (7,0)node[yshift=-0.5cm,xshift=0.5cm]{$A_4$}node[pointV]{} --++(1,0)node[pointV]{}--++(0.35,0.95)node[pointV]{}
      --++(-0.85,0.59)node[pointV]{}--++(-0.85,-0.59)node[pointV]{}
      --++(0.35,-0.95)
      --++(0.5,0.69)node[pointV]{}
      --+(0.5,-0.69)
      ++(0,0)--+(-0.85,0.26)
      --+(0,1.2)node[pointV]{}
      ++(0,0)--+(0.85,0.26)
      --+(0,1.2)
      ++(0,0)--+(0,0.85)
      ++(0,0) .. controls +(0.15,0.3) and +(0.15,-0.3) ..  +(0,1.2);

\draw (9,0)node[yshift=-0.5cm,xshift=0.5cm]{$A_5$}node[pointV]{} --++(1,0)node[pointV]{}--++(0.35,0.95)node[pointV]{}
      --++(-0.85,0.59)node[pointV]{}--++(-0.85,-0.59)node[pointV]{}
      --++(0.35,-0.95)
      --++(0.5,0.69)node[pointV]{}
      --+(0.5,-0.69)
      ++(0,0)--+(-0.85,0.26)
      ++(0,0)--+(0.85,0.26)
      ++(0,0)--+(0,0.85)
      ++(0,0)--+(0.7,0.85)node[pointV]{}
      ;

\draw (11,0)node[yshift=-0.5cm,xshift=0.5cm]{$A_6$}node[pointV]{} --++(1,0)node[pointV]{}--++(0.35,0.95)node[pointV]{}
      --++(-0.85,0.59)node[pointV]{}--++(-0.85,-0.59)node[pointV]{}
      --++(0.35,-0.95)
      --++(0.5,0.69)node[pointV]{}
      --+(0.5,-0.69)
      ++(0,0)--+(-0.85,0.26)
      ++(0,0)--+(0.85,0.26)
      ++(0,0)--+(0,0.85)
      ++(0,0)--++(0.7,0.85)node[pointV]{}
      --+(-0.7,0);
      
\draw (13,0)node[yshift=-0.5cm,xshift=0.5cm]{$A_7$}node[pointV]{} --++(1,0)node[pointV]{}--++(0.35,0.95)node[pointV]{}
      --++(-0.85,0.59)node[pointV]{}--++(-0.85,-0.59)node[pointV]{}
      --++(0.35,-0.95)
      --++(0.5,0.69)node[pointV]{}
      --+(0.5,-0.69)
      ++(0,0)--+(-0.85,0.26)
      ++(0,0)--+(0.85,0.26)
      ++(0,0)--+(0,0.85)--+(0.7,0.85)
      ++(0,0)+(0.7,0.85)node[pointV]{}
      --+(0.85,0.26);
      
%
\draw (0,-3.1)node[yshift=-0.5cm,xshift=0.5cm]{$B_1$}node[xshift=-0.2cm]{{\small $y$}}node[pointV]{} --++(1,0)node[pointV]{}--++(0.35,0.95)node[pointV]{}
      --++(-0.85,0.59)node[yshift=0.1cm,xshift=-0.2cm]{{\small $x$}}node[pointV]{}--++(-0.85,-0.59)node[pointV]{}
      --++(0.35,-0.95)
      --++(0.5,0.69)node[pointV]{}
      --+(0.5,-0.69)
      ++(0,0)--+(-0.85,0.26)
      ++(0,0)--+(0.85,0.26)
      ++(0,0)--+(0,0.85)--+(0.7,0.85)
      ++(0,0)--+(0.7,0.85)node[pointV]{}
      --+(0.85,0.26);
      
\draw (2,-3.1)node[yshift=-0.5cm,xshift=0.5cm]{$B_2$}node[xshift=-0.2cm]{{\small $y$}}node[pointV]{} --++(1,0)node[pointV]{}--++(0.35,0.95)node[pointV]{}
      --++(-0.85,0.59)node[yshift=0.1cm,xshift=-0.2cm]{{\small $x$}}node[pointV]{}--++(-0.85,-0.59)node[pointV]{}
      --++(0.35,-0.95)
      --++(0.5,0.69)node[pointV]{}
      --+(0.5,-0.69)
      ++(0,0)--+(-0.85,0.26)
      ++(0,0)--+(0.85,0.26)
      ++(0,0)--+(0,0.85)--+(0.7,0.85)
      ++(0,0)+(0.7,0.85)node[pointV]{}
      --+(0.85,0.26)
      plot [smooth] coordinates { +(-0.5,-0.69)  +(0.3,0)  +(0.7,0.85)};  
      
\draw (4,-3.1)node[yshift=-0.5cm,xshift=0.5cm]{$B_3$}node[xshift=-0.2cm]{{\small $y$}}node[pointV]{} --++(1,0)node[pointV]{}--++(0.35,0.95)node[pointV]{}
      --++(-0.85,0.59)node[yshift=0.1cm,xshift=-0.2cm]{{\small $x$}}node[pointV]{}--++(-0.85,-0.59)node[pointV]{}
      --++(0.35,-0.95)
      --++(0.5,0.69)node[pointV]{}
      --+(0.5,-0.69)
      ++(0,0)--+(-0.85,0.26)
      ++(0,0)--+(0.85,0.26)
      ++(0,0)--+(0,0.85)--+(0.7,0.85)
      ++(0,0)--+(0.7,0.85)node[pointV]{}
      --+(0.85,0.26)
      plot [smooth] coordinates { +(-0.5,-0.69)  +(0.3,0)  +(0.7,0.85)};

\draw (6,-3.1)node[yshift=-0.5cm,xshift=0.5cm]{$B_4$}node[xshift=-0.2cm]{{\small $y$}}node[pointV]{} --++(1,0)node[pointV]{}--++(0.35,0.95)node[pointV]{}
      --++(-0.85,0.59)node[yshift=0.1cm,xshift=-0.2cm]{{\small $x$}}node[pointV]{}--++(-0.85,-0.59)node[pointV]{}
      --++(0.35,-0.95)
      --++(0.5,0.69)node[pointV]{}
      --+(0.5,-0.69)
      ++(0,0)--+(-0.85,0.26)
      ++(0,0)--+(0.85,0.26)
      ++(0,0)--+(0,0.85)--+(0.7,0.85)--+(0.5,-0.69)
      ++(0,0)--+(0.7,0.85)node[pointV]{}
      --+(0.85,0.26);
      
\draw (8,-3.1)node[yshift=-0.5cm,xshift=0.5cm]{$B_5$}node[pointV]{} --++(1,0)node[pointV]{}--++(0.35,0.95)node[pointV]{}
      --++(-0.85,0.59)node[pointV]{}--++(-0.85,-0.59)node[pointV]{}
      --++(0.35,-0.95) 
      ++(0.5,0.69)node[pointV]{}
      --+(0.85,0.26)
      ++(0,0)--+(-0.85,0.26)
      ++(0,0)--++(0.5,-0.69)
      --++(-0.5,1.54)--++(-0.5,-1.54)
      ++(0.5,1.54)--++(0.7,0)node[pointV]{}node[pointV]{}--++(0.15,-0.59);

\draw (10,-3.1)node[yshift=-0.5cm,xshift=0.5cm]{$B_6$}node[xshift=-0.2cm]{{\small $y$}}node[pointV]{} --++(1,0)node[pointV]{}--++(0.35,0.95)node[pointV]{}
      --++(-0.85,0.59)node[yshift=0.1cm,xshift=-0.2cm]{{\small $x$}}node[pointV]{}--++(-0.85,-0.59)node[pointV]{}
      --++(0.35,-0.95)
      --++(0.5,0.69)node[pointV]{}
      --+(0.5,-0.69)
      ++(0,0)--+(0,0.85)
      +(0,-0.4)node[pointV]{}--+(-0.85,0.26)
      +(0,-0.4)--+(0.85,0.26)
      +(0,-0.4)--+(-0.5,-0.69)
      +(0,-0.4)--+(0.5,-0.69);

\draw (12,-3.1)node[yshift=-0.5cm,xshift=0.5cm]{$B_7$}node[xshift=-0.2cm]{{\small $y$}}node[pointV]{} --++(1,0)node[pointV]{}--++(0.35,0.95)node[pointV]{}
      --++(-0.85,0.59)node[yshift=0.1cm,xshift=-0.2cm]{{\small $x$}}node[pointV]{}--++(-0.85,-0.59)node[pointV]{}
      --++(0.35,-0.95)
      --++(0.5,0.69)node[pointV]{}
      --+(0.5,-0.69)
      ++(0,0)--+(-0.85,0.26)
      ++(0,0)--+(0,0.85)
      +(0,-0.4)node[pointV]{}--+(-0.85,0.26)
      +(0,-0.4)--+(0.85,0.26)
      +(0,-0.4)--+(-0.5,-0.69)
      +(0,-0.4)--+(0.5,-0.69);
      
%
%
%
\draw (0,-6.2)node[yshift=-0.5cm,xshift=0.5cm]{$C_1$}node[xshift=-0.2cm]{{\small $y$}}node[pointV]{} --++(1,0)node[pointV]{}--++(0.35,0.95)node[pointV]{}
      --++(-0.85,0.59)node[yshift=0.1cm,xshift=-0.2cm]{{\small $x$}}node[pointV]{}--++(-0.85,-0.59)node[pointV]{}
      --++(0.35,-0.95)
      --++(0.5,0.69)node[pointV]{}
      --+(0.5,-0.69)
      ++(0,0)--+(-0.85,0.26)
      ++(0,0)--+(0.85,0.26)
      ++(0,0)--+(0,0.85)
      +(0,-0.4)node[pointV]{}--+(-0.85,0.26)
      +(0,-0.4)--+(0.85,0.26)
      +(0,-0.4)--+(-0.5,-0.69)
      +(0,-0.4)--+(0.5,-0.69);

\draw (2,-6.2)node[yshift=-0.5cm,xshift=0.5cm]{$C_2$}node[xshift=-0.2cm]{{\small $y$}}node[pointV]{} --++(1,0)node[pointV]{}--++(0.35,0.95)node[pointV]{}
      --++(-0.85,0.59)node[yshift=0.1cm,xshift=-0.2cm]{{\small $x$}}node[pointV]{}--++(-0.85,-0.59)node[pointV]{}
      --++(0.35,-0.95)
      --++(0.5,0.69)node[pointV]{}
      --+(0.5,-0.69)
      ++(0,0)--+(-0.85,0.26)
      ++(0,0)--+(0.85,0.26)
      ++(0,0)--+(0,0.85)
      --+(0,-0.4)node[pointV]{}--+(-0.85,0.26)
      +(0,-0.4)--+(0.85,0.26)
      +(0,-0.4)--+(-0.5,-0.69)
      +(0,-0.4)--+(0.5,-0.69);
      
\draw (4,-6.2)node[yshift=-0.5cm,xshift=0.5cm]{$C_3$}node[xshift=-0.2cm]{{\small $y$}}node[pointV]{} --++(1,0)node[pointV]{}--++(0.35,0.95)node[pointV]{}
      --++(-0.85,0.59)node[yshift=0.1cm,xshift=-0.2cm]{{\small $x$}}node[pointV]{}--++(-0.85,-0.59)node[pointV]{}
      --++(0.35,-0.95)
      --++(0.5,0.69)node[pointV]{}
      --+(0.5,-0.69)
      ++(0,0)--+(-0.85,0.26)
      ++(0,0)--+(0.85,0.26)
      ++(0,0)--+(0,0.85)
      --+(0,-0.4)node[pointV]{}--+(-0.85,0.26)
      +(0,-0.4)--+(0.85,0.26)
      +(0,-0.4)--+(-0.5,-0.69)
      +(0,-0.4)--+(0.5,-0.69)
      ++(0,0)+(-0.85,0.26)
      --+(0.85,0.26);

\draw (6,-6.2)node[yshift=-0.5cm,xshift=0.5cm]{$C_4$}node[pointV]{} --++(1,0)node[pointV]{}--++(0.35,0.95)node[pointV]{}
      --++(-0.85,0.59)node[pointV]{}--++(-0.85,-0.59)node[pointV]{}
      --++(0.35,-0.95)
      --++(0.5,0.69)node[pointV]{}
      --+(0.5,-0.69)
      ++(0,0)--+(-0.85,0.26)
      ++(0,0)--+(0.85,0.26)
      ++(0,0)--+(0,0.85)
      +(0,-0.4)node[pointV]{}--+(-0.85,0.26)
      +(0,-0.4)--+(0.85,0.26)
      +(0,-0.4)--+(-0.5,-0.69)
      +(0,-0.4)--+(0.5,-0.69)
      +(0,-0.4) .. controls +(0.2,0.3) and +(0.2,-0.5) .. +(0,0.85);

\draw (8,-6.2)node[yshift=-0.5cm,xshift=0.5cm]{$C_5$}node[pointV]{} --++(1,0)node[pointV]{}--++(0.35,0.95)node[pointV]{}
      --++(-0.85,0.59)node[pointV]{}--++(-0.85,-0.59)node[pointV]{}
      --++(0.35,-0.95)
      --++(0.5,0.69)node[pointV]{}
      --+(0.5,-0.69)
      ++(0,0)--+(-0.85,0.26)
      ++(0,0)--+(0.85,0.26)
      ++(0,0)--+(0,0.85)
      --+(0,-0.4)node[pointV]{}--+(-0.85,0.26) 
      +(0,-0.4)--+(0.85,0.26)
      +(0,-0.4)--+(-0.5,-0.69)
      +(0,-0.4)--+(0.5,-0.69)
      +(0,-0.4) .. controls +(0.2,0.3) and +(0.2,-0.5) .. +(0,0.85);

\draw (10,-6.2)node[yshift=-0.5cm,xshift=0.5cm]{$C_6$}node[xshift=-0.2cm]{{\small $y$}}node[pointV]{} --++(1,0)node[pointV]{}--++(0.35,0.95)node[pointV]{}
      --++(-0.85,0.59)node[yshift=0.1cm,xshift=-0.2cm]{{\small $x$}}node[pointV]{}--++(-0.85,-0.59)node[pointV]{}
      --++(0.35,-0.95)
      --++(0.5,0.69)node[pointV]{}
      --+(0.5,-0.69)
      ++(0,0)--+(-0.85,0.26)
      ++(0,0)--+(0.85,0.26)
      ++(0,0)--+(0,0.85)
      ++(0,0)--+(0,-0.4)node[pointV]{}
      +(0,-0.4)--+(-0.5,-0.69)
      +(0,-0.4)--+(0.5,-0.69)
      ++(0,0)+(-0.85,0.26)
      --+(0.85,0.26)
      +(0,-0.4) .. controls +(0.2,0.3) and +(0.2,-0.5) .. +(0,0.85);


%
%
 \draw (0-0.2,-9.5) node[yshift=-0.5cm,xshift=1cm]{$D_1$} node[pointV]{}-- ++(0.5774,1) node[pointV]{}--++(0.5774,1)  node[pointV]{}
                   --++(0.5774,-1)  node[pointV]{}--++(0.5774,-1)  node[pointV]{}
                   --++(-2*0.5774,0)--++(-0.5774,1)--++(2*0.5774,0)
                   --++(-0.5774,-1)node[pointV]{}--+(-2*0.5774,0) 
                   ++(0,0)--++(0,2/3)node[pointV]{} --+(-0.5774,1/3)
                   ++(0,0)--+(0.5774,1/3) ;                  
                   
\draw (3-0.2,-9.5) node[yshift=-0.5cm,xshift=1cm]{$D_2$} node[xshift=-0.2cm]{{\small $y$}}node[pointV]{}-- ++(0.5774*4,0) node[pointV]{}--++(-0.5774*2,2) node[yshift=0.1cm,xshift=-0.2cm]{{\small $x$}} node[pointV]{}
                   --++(-0.5774*2,-2) --++(0.5774*2,2/3)  node[pointV]{}
                   --++(0.5774,1/3) node[pointV]{}--++(-0.5774*2,0) node[pointV]{}
                   --++(0.5774,-1) node[pointV]{}--+(0.5774,1)
                   ++(0,0)--+(0,2)
                   +(0.5774*2,0)--+(-0.5774,1);

\draw (6-0.2,-9.5) node[yshift=-0.5cm,xshift=1cm]{$D_3$} node[xshift=-0.2cm]{{\small $y$}}node[pointV]{}-- ++(0.5774*4,0) node[pointV]{}--++(-0.5774*2,2)  node[yshift=0.1cm,xshift=-0.2cm]{{\small $x$}}node[pointV]{}
                   --++(-0.5774*2,-2) --++(0.5774*2,2/3)  node[pointV]{}
                   --++(-0.5774,1/3) node[pointV]{}
                   --++(0.5774,-1) node[pointV]{}--++(0.5774,1) node[pointV]{}--++(-0.5774*2,0) node[pointV]{}
                   ++(0.5774,-1/3)--+(0,4/3)
                   ++(0,0)--++(0.5774*2,-2/3);

\draw (9-0.2,-9.5) node[yshift=-0.5cm,xshift=1cm]{$D_4$};

\foreach \r in {0, 120, 240}{
\draw[shift={(9-0.2,-9.5)},rotate around={\r:(0.5774*2,2/3)}] (0,0)  node[pointV]{} --+(0.5774*2,2/3)  node[pointV]{} --+(1.1,0.3)  node[pointV]{}-- ++(0,0)
                   --++(0.5774*4,0)  node[pointV]{};

}

\draw (12-0.2,-9.5)node[yshift=-0.5cm,xshift=1cm]{$D_5$} node[pointV]{}-- ++(2,0) node[pointV]{}--++(0,2) node[pointV]{}
       -- ++(-2,0) node[pointV]{}--++(0,-2)  node[pointV]{}
       --+(0.8,0.8) node[pointV]{}--+(0.3,0.8) node[pointV]{}--+(0,0)
       --+(0.8,0.3) node[pointV]{}--+(2,0) --+(0.8,0.8) --+(0.8,0.3)
       +(0.8,0.8)--+(2,2)
       +(2,0)--+(0,2)
       ;   

\end{tikzpicture}
\caption{The 25  non-isomorphic non-word-representable graphs on 7 vertices}\label{25-non-repr}
\end{center}
\end{figure}

We do not know whether $\mathcal{G}^{(1)}$ coincides with the class of all graphs, but at least we can show that all 25 graphs in Figure~\ref{25-non-repr} are 1-$11$-representable, which we do next. We will use the fact that all graphs on at most six vertices are 1-$11$-representable, which follows from the proof of Theorem~\ref{strict-inclusion-thm}, where we 1-11-represent the only non-word-representable graph on six vertices.

The graphs $A_1$ and $A_5$ are 1-$11$-representable by Theorem~\ref{degree-1-thm}, since they have a vertex of degree 1. Theorem~\ref{same-neighbourhood-thm} can be applied to the graphs $A_4$, $C_4$ and $C_5$ since each of these graphs have a pair of vertices whose neighbourhoods are the same up to removing these vertices. Further, Theorem~\ref{add-triangle-thm} gives 1-$11$-representability of the graphs $A_6$, $A_7$, $B_5$, $D_1$, $D_4$ and $D_5$ since each of these graphs has a triangle with a vertex of degree 2. Explicit easy-to-check 1-$11$-representants of the graphs $A_2$ and $A_3$ are, respectively, 437257161521324354 and 437251761521324354. For each graph $G$ of the remaining 12 graphs in Figure~\ref{25-non-repr}, we provide vertices $x$ and $y$ connecting which by an edge results in a word-representable graph $G_{xy}$, so that Theorem~\ref{edge-thm} can be applied (removing the edge $xy$ from $G_{xy}$) to see that $G$ is 1-$11$-representable. The fact that $G_{xy}$ is word-representable follows from it not being isomorphic to any of the graphs in Figure~\ref{25-non-repr}, where all non-word-representable graphs on seven vertices are presented. Alternatively, one can use the software packages \cite{Glen,Hans} to see that $G_{xy}$ is word-representable  (the software can produce an easy to check word representing $G_{xy}$).

\section{All graphs are $2$-$11$-representable}\label{sec-2-11-representable}
	A simple graph $G=(V,E)$ is {\em permutationally $k$-$11$-representable} if there is a $k$-$11$-representant $w$ of $G$ which is a concatenation of permutations of $V$.
	Such a word $w$ is called a {\em permutational $k$-$11$-representant} of $G$.
	In this section, we will prove that every graph is permutationally $2$-$11$-representable, by an inductive construction of a $2$-$11$-representant.	This result also implies that every graph is permutationally $k$-$11$-representable for any integer $k \geq 2$ by Lemma~\ref{extension-lem}.
	We still do not know whether every graph is $1$-$11$-representable or not.
	
	\medskip
	
	We begin with a simple observation.
	\begin{obs}\label{duplicate}
	If a word $w = w_1 P w_2$ $k$-$11$-represents a graph $G$ where $P$ is a permutation of $V(G)$, then the word $w' = w_1 P P w_2$ also $k$-$11$-represents the graph $G$.
	\end{obs}

In the proof of the following theorem, we use the following notation. For a pair of vertices $u$ and $v$ in $G$, we write $u \sim v$ if $u$ and $v$ are adjacent in $G$, and $u \nsim v$ otherwise. Also, for convenience, we separate permutations in a permutational $2$-$11$-representant by space.

	\begin{theorem}\label{all2-11}
	Let $G = (V,E)$ be a graph on $n \geq 2$ vertices.
	Then there is a permutational $2$-$11$-representant $w$ over the alphabet $V$ such that
		\begin{enumerate}[(a)]
		\item\label{aaa} $w$ is a concatenation of at most $f(n) = n^2 - n + 2$ permutations of $V$, and
		\item\label{bbb} for each $i \in V$, there exists a permutation $P$ in  $w$ that starts with $i$, i.e. $P = i Q$ where $Q$ is a permutation of $V \setminus \{i\}$. 
		\end{enumerate}
	\end{theorem} 
	\begin{proof}
	We use induction on $n$. 
	For the base case when $n = 2$, we take the $2$-$11$-representant $12 ~21 ~12 ~12$ of a complete graph on the vertex set $[2]$, and we take the $2$-$11$-representant $12 ~21 ~12 ~21$ of two isolated vertices $1$ and $2$.
	
	Suppose $n \geq 3$ and let $G$ be a graph with the vertex set  $V=[n]$.
	By the induction hypothesis, the induced subgraph $G\setminus n$ can be $2$-$11$-represented by a word 
	$$P_1 P_2 \cdots P_{f(n-1)},$$ 
	where each $P_i$ is a permutation of $[n-1]$.
	Note that, by the condition (\ref{bbb}), for each $i \in [n-1]$ we can choose one $k_i \in [f(n-1)]$ so that the permutation $P_{k_i} = i Q_i$ where $Q_i$ is a permutation of $[n-1]\setminus\{i\}$.
	
	Now we construct a $2$-$11$-representant $w$ of $G$ satisfying the conditions (\ref{aaa}) and (\ref{bbb}).
	For each $i \in [f(n-1)]$, let 
	$$P_i' := \begin{cases}j n Q_j & \;\text{if }i = k_j\text{ for some }j\in[n-1]\text{ and }n\sim j \\ j n Q_j  ~  n j Q_j  ~  j n Q_j & \;\text{if }i = k_j\text{ for some }j\in[n-1]\text{ and }n\nsim j \\ n P_i & \;\text{otherwise},\end{cases}$$
	and define $$w = P_1' \cdots P_{f(n-1)}'.$$
Note that all of $n P_i, j n Q_j$ and $n j Q_j$ are permutations of $[n]$. Thus, the word $w$ is a concatenation of at most $f(n-1) + 2(n-1) = n^2 - n + 2$ permutations of $[n]$ and the condition (\ref{bbb}) obviously holds.
	It remains to show that $w$ $2$-$11$-represents the graph $G$.
	
	By applying Observation~\ref{duplicate} repeatedly, we observe that the subword of $w$ induced by $[n-1]$ $2$-$11$-represents the graph $G \setminus n$.
	Hence, it is sufficient to check whether the adjacency of the vertex $n$ is preserved.
	For each $i \in [n-1]$, the subword of $w$ induced by the letters $i$ and $n$ is given by $$(ni)^a ~(in)~ (ni)^b \;\;\text{ if }\;\; i \sim n,$$ having at most $2$ occurrences of the consecutive pattern $11$, and is given by $$(ni)^a ~in ~ni ~in ~(ni)^b \;\;\text{ if }\;\; i \nsim n$$ with at least $3$ occurrences of the consecutive pattern $11$.
	This completes the proof.
	\end{proof}

	Recall that, by Theorem~\ref{connected-thm}, a graph $G$ is $2$-$11$-representable if and only if each connected component of $G$ is $2$-$11$-representable.
	It is obvious that every connected graph $G$ on at least two vertices contains no isolated vertex, and that there always exists a vertex $v$ in $G$ such that $G\setminus v$ is again connected.
	Applying this observation, we can improve the function $f(n)$ in Theorem~\ref{all2-11} for connected graphs as follows.
	\begin{theorem}\label{improved-length-thm}
	Let $G = (V,E)$ be a connected graph on $n \geq 2$ vertices.
	Then there is a $2$-$11$-representant $w$ over the alphabet $V$ such that
		\begin{enumerate}[(a)]
		\item $w$ is a concatenation of $f(n) = n^2-3n+4$ permutations of $V$, and
		\item for each $i \in V$, there is a permutation $P$ in the word $w$ which starts with $i$.
		\end{enumerate}
	\end{theorem}

\section{Open problems on $k$-$11$-representable graphs}\label{open-prob-sec}

The most intriguing open question in the theory of $k$-$11$-representable graphs is the following.

\begin{problem}\label{prob1} {\em Is it true that any graph is $1$-$11$-representable? If not, then which classes of graphs are $1$-$11$-representable? In particular, are all planar graphs or all $4$-chromatic graphs $1$-$11$-representable?}\end{problem}  

Note that $W_5$ shows that not all planar or $4$-chromatic graphs are word-representable, justifying our specific interest to these classes of graphs.

By Theorem~\ref{unif-wr}, any word-representable graph can be represented by a uniform word. It is known that the class of perutationally word-representable graphs coincides with the class of comparability graphs \cite{KP}.
Also, by Theorem~\ref{all2-11} and Lemma~\ref{extension-lem},  for every $k \geq 2$ any graph is permutationally $k$-$11$-representable.
Thus, the following questions are natural. 

\begin{problem}\label{prob3} {\em  Is it true that any $1$-$11$-representable graph can be represented by a concatenation of permutations? Or, at least, by a uniform word?} \end{problem}  


It is known~\cite{HKP} that if a graph $G$ with $n$ vertices is word-representable, then it can be represented by a uniform word of length at most $2n(n-\kappa)$  where $\kappa$ is the size of a maximum clique in $G$. An upper bound for the length of $k$-$11$-representants for $k\geq 2$ can be derived from Theorems~\ref{all2-11} and \ref{improved-length-thm} and Lemma~\ref{extension-lem}. In particular, $2$-$11$-representants are of length $O(n^3)$. However, we have no upper bounds for the length of words $1$-$11$-representing graphs. 

\begin{problem}\label{prob5}  {\em  Provide an upper bound for the length of words $1$-$11$-representing graphs.}  \end{problem}  

Remind that Theorem~\ref{interval-thm} shows that the class of interval graphs is precisely the class of $1$-$11$-representable graphs that can be represented $2$-uniformly.

\begin{problem}\label{prob6}  {\em  Does the class of $m$-uniformly $1$-$11$-representable graphs, for $m\geq 3$, have any interesting/useful properties? In particular, is there a description of such graphs in terms of forbidden subgraphs? A good starting point to answer the last question should be the case of $m=3$.}  \end{problem}  

	As the first step in the direction of Problem~\ref{prob6}, we discuss a geometric realization of $r$-uniformly $k$-$11$-representable graphs, which might give new results on characterization problems for word-representable graphs.
	This is motivated from the fact that a graph is $2$-uniformly $0$-$11$-representable if and only if it is a circle graph.

	Take a convex curve $\gamma = \gamma(t)$, $t \in [0,1]$, in the plane (not necessarily closed) and consider a set of $n\times r$ distinct real numbers $S=\{x_1, x_2, \ldots, x_{nr}\}$ such that $$0\leq x_1 < x_2 < \cdots < x_{nr} \leq 1.$$
	Color each element in $S$ by $[n]$, i.e. we choose an injection $\phi:S\to[n]$, such that $|S_i|=r$ where $S_i := \{x\in S:\phi(x)=i\}$.
	Say $S_i := \{x_{i_1}x_{i_2},\ldots,x_{i_r}\}$ where $x_{i_j}<x_{i_k}$ whenever $j<k$.
	Now we draw $n$ piecewise linear convex curves
	$$C_i:\; \gamma(x_{i_1})-\gamma(x_{i_2})-\cdots-\gamma(x_{i_r}),\;\;i\in[n]$$
	connecting the points $\gamma(x_{i_1}),\ldots,\gamma(x_{i_r})$.
	See Figure~\ref{piecewiselinear} for an illustration.
	
	\begin{figure}[htbp]
\centering
\begin{tikzpicture}[main node/.style={fill,circle,draw,inner sep=0pt,minimum size=4pt}, scale=2]

\begin{scope}[xshift=0cm, yshift=0cm]
\draw (0:2) arc (0:180:2);
\node[main node, label=left:$\gamma(t_{i,1})$] (t1) at (170:2){};
\node[main node, label=left:$\gamma(t_{i,2})$] (t2) at (140:2){};
\node[main node, label=above:$\gamma(t_{i,3})$] (t3) at (110:2){};
\node[main node, label=right:$\gamma(t_{i,r-1})$] (t4) at (50:2){};
\node[main node, label=right:$\gamma(t_{i,r})$] (t5) at (20:2){};
\node[label=right:{{\Large $\gamma$}}] at (5:2){};
\node[label=above:{$\cdots$}] at (77.5:2){};
\draw[thick] (t1)--(t2)--(t3);
\draw[thick, dotted] (t3)--(80:2)--(t4); \draw[thick] (t4)--(t5);
\node at (2.6,0){};
\end{scope}
\end{tikzpicture}
\caption{The piecewise linear convex curve $C_i$}
\label{piecewiselinear}
\end{figure}
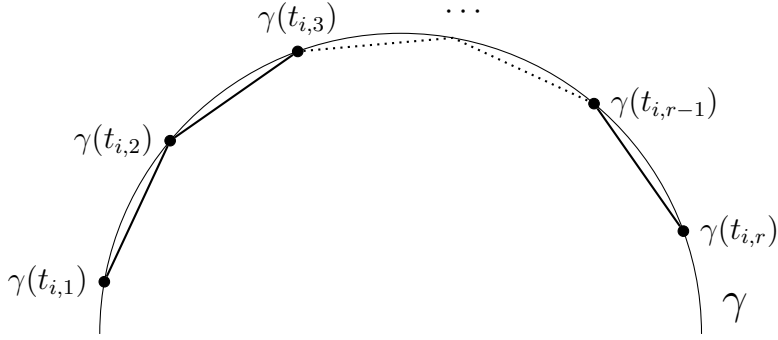

	Note that for every two distinct $C_i$ and $C_j$, we have $|C_i \cap C_j| \leq 2r-3$.
	For each $m \in [2r-3]$, we define {\em $m$-intersection graph} $I_m(\mathcal{C}\phi)$ of $\mathcal{C}_\phi = \{C_1,\ldots, C_n\}$ as the graph on $[n]$ such that two distinct vertices $i, j\in[n]$ are adjacent if and only if $|C_i \cap C_j| \geq m$.
	In particular, when $m = 1$, the graph $I_1(\mathcal{C}\phi)$ is just the intersection graph of $\mathcal{C}$.

	On the other hand, regarding $[n]$ as an alphabet, we construct a word over $[n]$:
	$$w = \phi(x_1)\phi(x_2)\ldots\phi(x_{nr}).$$
	Then the subword of $w$ induced by two distinct letters $i$ and $j$ has at most $2r-3-m$ occurences of the consecutive pattern $11$ if and only if $i\sim j$ (i.e.\ $ij$ is an edge) in the graph $I_m(\mathcal{C})$.
	As an immediate consequence of this relation, we observe the following.
	\begin{prop}
	For every positive integer $m$ and $r\geq2$ such that $1\leq m\leq 2r-3$, a graph $G$ is $r$-uniform $(2r-3-m)$-$11$-representable if and only if there exists a coloring $\phi:S\to[n]$ so that $G = I_m(\mathcal{C}_\phi)$.
	\end{prop}
	\begin{proof}
	By the above argument, it is sufficent to prove that every $r$-uniform $(2r-3-m)$-$11$-representable graph $G$ assigns a coloring $\phi:S\to[n]$ so that $G = I_m(\mathcal{C}_\phi)$.
	This is obvious since an $r$-uniform $(2r-3-m)$-$11$-representant of $G$ naturally gives a coloring $\phi:S\to[n]$ with $|S_i| = r$, and the corresponding family $\mathcal{C}_\phi$ of $n$ piecewise linear convex curves satisfies that $I_m(\mathcal{C}_\phi)= G$.
	\end{proof}
	Note that for every $k > 2r-3$, the only $r$-uniform $k$-$11$-representable graphs are complete graphs.
	When $m = 0$, the $r$-unifrom $(2r-3)$-$11$-representable graphs can be specified as a well-known graph class.
	\begin{prop}
	A graph is $r$-uniform $(2r-3)$-$11$-representable if and only if it is an interval graph.
	\end{prop}
	\begin{proof}
	Given an $r$-unifrom $(2r-3)$-$11$-representable graph $G$ on $[n]$, take any $r$-unifrom $(2r-3)$-$11$-representant $w$.
	Clearly two vertices $i$ and $j$ are not adjacent in $G$ if and only if the subword of $w$ induced by $i$ and $j$ consists of $r$ consecutive $i$'s and $r$ consecutive $j$'s, i.e. either $i\ldots i~j\ldots j$ or $j\ldots j~i\ldots i$.
	Embed the word $w$ on a line, and consider an interval $J_i$ defined by the leftmost $i$ and the rightmost $i$.
	Then $G$ is the intersection graph of $\{J_1,\ldots,J_n\}$.

	For the other direction, let $G$ be the intersection graph of intervals $\{J_1,\ldots,J_n\}$.
	Since $n$ is finite, we may assume that each $J_i$ is bounded and no two intervals share an end-point.
	We label the end-points of $J_i$ by $i$, and construct a $2$-uniform word by reading the labels from left to right.
	Then we insert $r-2$ $i$'s in arbitrary positions between two original $i$'s in the word.
	This gives us a $r$-unifrom $(2r-3)$-$11$-representant of $G$.
	\end{proof}

\section*{Acknowledgments} This work was supported by the National Research Foundation of Korea (NRF) grant funded by the Korea government (MSIP) (2016R1A5A1008055) and the Ministry of Education of Korea (NRF-2016R1A6A3A11930452). The work of the last author was supported by the program of fundamental scientific researches of the SB RAS I.5.1., project 0314-2016-0014.


\begin{thebibliography}{20}
\bibitem{Bell} E. J. L. Bell. Word-graph theory. PhD thesis, Lancaster University, 2011.
\bibitem{BBPP} I. M. Bomze, M. Budinich, P. M. Pardalos, M. Pelillo. ``The maximum clique problem'', Handbook of Combinatorial Optimization, 4, Kluwer Academic Publishers (1999) 1--74.
\bibitem{Cer} J. \u{C}ern\'{y}. Coloring circle graphs. {\em Electronic Notes in Discr. Math.} {\bf 29} (2007) 457--461.
\bibitem{Danielsen} L. E. Danielsen. Database of Circle Graphs. Available at \\ http://www.ii.uib.no/\verb>~>larsed/circle/
\bibitem{Glen} M. Glen. Software to work with word-representable graphs. Available at \\ https://personal.cis.strath.ac.uk/sergey.kitaev/word-representable-graphs.html.
\bibitem{GZ} R. Graham, N. Zang. Enumerating split-pair arrangements,
\emph{J. Combin. Theory, Series A} {\bf 115}, Issue 2 (2008), 293--303.
\bibitem{HKP} M. Halld\'orsson, S. Kitaev, A. Pyatkin. Semi-transitive orientations and word-representable graphs. {\em Discr. Appl. Math.} {\bf 201} (2016) 164--171.
\bibitem{K} S. Kitaev. A comprehensive introduction to the theory of word-representable graphs, {\em Lecture Notes in Computer Science} {\bf 10396} (2017) 36--67.
\bibitem{Kit1} S. Kitaev. Existence of $u$-representation of graphs, {\em Journal of Graph Theory} {\bf 85} (2017) 3, 661--668.
\bibitem{KL} S. Kitaev, V. Lozin. Words and graphs. Springer, 2015.
\bibitem{KP} S. Kitaev, A. Pyatkin. On representable graphs. {\em
J. Autom., Lang. and Combin.} {\bf 13} (2008) 1, 45--54.
\bibitem{KS} S. Kitaev, S. Seif. Word problem of the Perkins semigroup
via directed acyclic graphs. {\em Order} {\bf 25} (2008) 3, 177--194.
\bibitem{Koebe} M. Koebe. On a new class of intersection graphs. In Jaroslav Ne\u{s}et\u{r}il and Miroslav Fiedler, editors, Fourth Czechoslovakian Symposium on
Combinatorics, Graphs and Complexity (Prachatice, 1990), volume 51 of Ann. Discrete Math., pages 141--143. North-Holland, Amsterdam, 1992.
\bibitem{Lov} L. Lov\'{a}sz. Perfect graphs. Selected Topics in Graph Theory, 2, London: Academic Press (1983) 55--87.
\bibitem{Remmel} Jeff Remmel. Private communication, 2017.
\bibitem{oeis} N.~J.~A.~Sloane, The on-line encyclopedia of integer sequences,
published electronically at \phantom{*} {\tt
http://oeis.org}.
\bibitem{Hans} H. Zantema. Software to work with word-representable graphs. Available at \ http://www.win.tue.nl/\verb>~>hzantema/reprnr.html
\end{thebibliography}
\end{document}